\documentclass{amsart}
\usepackage[left=2.65cm,right=2.65cm,top=2.7cm,bottom=2.7cm]{geometry}
\usepackage{amsmath,amssymb}
\usepackage{graphicx,tikz,mathtools}
\newtheorem{theorem}{Theorem}
\newtheorem{proposition}[theorem]{Proposition}

\newtheorem{lemma}[theorem]{Lemma}

\theoremstyle{definition}
\newtheorem{definition}[theorem]{Definition}
\newtheorem*{definition*}{Definition}

\theoremstyle{remark}
\newtheorem{remark}[theorem]{Remark}

\newcommand{\la}{\lambda}

\def\RR{\mathbb{R}}

\def\ZZ{\mathbb{Z}}

\def\TT{\mathbb{T}}

\newcommand{\cV}{{\mathcal V}}

\newcommand{\pd}{\partial}
\newcommand\minus\backslash

\newcommand{\abs}[1]{\left|#1\right|}

\newcommand\lan\langle
\newcommand\ran\rangle

\DeclareMathOperator\Div{div} 

\DeclareMathOperator\End{End}

\renewcommand\leq\leqslant
\renewcommand\geq\geqslant
\newlength{\intwidth}

 \DeclareMathOperator\curl{curl}

\DeclareMathOperator\Tr{tr}
\DeclareMathOperator\sym{Sym}
\DeclareMathOperator\id{Id}
\newcommand\vol{\mathrm{Vol}}
\newcommand\Ker{\mathrm{ker}}

\begin{document}

\title[Isolated steady solutions of 3D Euler]{Isolated steady solutions\\ of the 3D  Euler equations}

 \author{Alberto Enciso}
 \address{Instituto de Ciencias Matem\'aticas, Consejo Superior de
   Investigaciones Cient\'\i ficas, 28049 Madrid, Spain}
 \email{aenciso@icmat.es}

 \author{Willi Kepplinger}
 \address{Department of Mathematics, University of Vienna, Vienna, Austria}
 \email{willi.kepplinger@univie.ac.at}

 \author{Daniel Peralta-Salas}
 \address{Instituto de Ciencias Matem\'aticas, Consejo Superior de
   Investigaciones Cient\'\i ficas, 28049 Madrid, Spain}
 \email{dperalta@icmat.es}

%
%
\begin{abstract}
We show that there exist closed three-dimensional Riemannian manifolds where the incompressible Euler equations exhibit smooth steady solutions that are isolated in the $C^1$-topology. The proof of this fact combines ideas from dynamical systems, which appear naturally because these isolated states have strongly chaotic dynamics, with techniques from spectral geometry and contact topology, which can be effectively used to analyze the steady Euler equations on carefully chosen Riemannian manifolds. Interestingly, much of this strategy carries over to the Euler equations in Euclidean space, leading to the weaker result that there exist analytic steady solutions on~$\TT^3$ such that the only analytic steady Euler flows in a $C^1$-neighborhood must belong to a certain linear space of dimension six. For comparison, note that in any $C^k$-neighborhood of a shear flow there are infinitely many linearly independent analytic shears.
\end{abstract}

\maketitle

\section{Introduction}

In the study of steady states of the incompressible Euler equations, one often finds a subtle interplay between rigidity and flexibility properties, that is, between the existence of a wealth of steady Euler flows and the significant constraints that they must nonetheless satisfy. To elaborate on this idea, for the time being we shall focus on the two dimensional case, where our understanding is much more complete.

As is well known, in~$\RR^2$, the velocity field can be written in terms of the perpendicular gradient of the stream function as $v:=\nabla^\perp\psi=(\pd_2\psi,-\pd_1\psi)$, whereas in all dimensions the pressure is determined by the velocity through the formula $p:=-\Delta^{-1}\Div(v\cdot\nabla v)$. The steady Euler equations
\begin{equation}\label{E.Euler}
	v\cdot \nabla v=-\nabla p\,,\qquad \Div v=0\,,
\end{equation}
can then be equivalently written as
\begin{equation*}
\nabla^\perp \psi\cdot\nabla\Delta\psi=0\,.
\end{equation*}

While this is generally not a well behaved equation from the point of view of the analysis of PDEs, it does certainly ensure the existence of many steady states. Specifically, on the plane~$\RR^2$ (or on~$\TT^2$ in the case of periodic conditions, with $\TT:=\RR/2\pi\ZZ$, or more generally on any two-dimensional Riemannian manifold), any function satisfying an elliptic equation of the form
\begin{equation}\label{E.DeltaF}
\Delta\psi = f(\psi)
\end{equation}
defines a steady Euler flow as $v:=\nabla^\perp \psi$. Note that not all steady states can be constructed this way, as the vorticity~$\Delta\psi$ and the stream function do not generally satisfy an equation of this form globally.

It is a truly infinite-dimensional feature of the Euler equations that steady solutions of the form~\eqref{E.DeltaF} are often embedded in rich families of solutions. Indeed, Constantin, Drivas and Ginsberg~\cite{CDG} have recently shown that, under suitable technical hypotheses, these solutions are structurally stable. Also, given a steady state on an annular domain for which the linearization of~\eqref{E.DeltaF} is positive definite, Choffrut and \v{S}ver\'ak~\cite{ChofS} completely characterized the nearby steady states  by showing that they are in one-to-one correspondence with their distribution functions, so that there exists a unique stationary solution on the corresponding orbit in the group of area-preserving diffeomorphisms.

On the rigidity side, it is known that, under suitable technical hypotheses, steady Euler flows must inherit the symmetries of their vessel, understood as a two-dimensional Riemannian manifold possibly with boundary. This is the case when the flow does not have any stagnation points if the domain is a disk, an annulus or a periodic channel~\cite{Hamel,Hamel1, Hamel2} or if the solution is compactly supported and satisfies certain sign conditions~\cite{gomez,R}. In the particular case of the periodic channel $\TT\times\RR$, for instance, this result ensures that the only steady Euler flows without stagnation points are shear flows $v(x):=(V(x_2),0)$. Similar results also hold when the steady flow
is dynamically stable, satisfies~\eqref{E.DeltaF} and the linearized operator is positive~\cite{CDG}.

Results such as~\cite{Hamel,Hamel1, Hamel2} ensure that steady Euler flows without stagnation points are isolated from non-shears in the $C^1$-norm. The dependence of this property on the absence of stagnation points and on the topology is remarkable~\cite{CZ,LZ}; for instance~\cite{CZ}, the vanishing shear flow $(\sin x_2,0)$ is not isolated from non-shears even in the analytic topology, while the shear flow $(x_2^2,0)$, which also has stagnation points, is isolated from non-shears in~$H^s$, $s>5$.

In three dimensions, the interplay between rigidity and flexibility is similar, but our understanding is much more limited. The steady Euler equations can be equivalently written as
\begin{equation}\label{E.Euler2}
v\times\curl v=\nabla F\,,
\end{equation}
where the so-called {\em Bernoulli function}~$F$ is arbitrary,
and the closest analog to~\eqref{E.DeltaF} is
\[
\curl v= f v\,,\qquad \Div v=0\,.
\]
Although these equations do have a rich space of solutions, and the fact that~$f$ is arbitrary provides some additional flexibility, whenever the function~$f$ is nonconstant the equations exhibit some surprising rigidity phenomena~\cite{rare,notsorare}. As a rule of thumb, in 3D, an advantage is that divergence-free fields can have    a richer behavior, and a drawback is that the equations are notoriously harder to analyze.

A recent survey by Drivas and Elgindi~\cite{DE} draws the attention to the fact that none of the known steady 2D Euler flows are isolated in~$C^1$. Note this is also true for the shear flows studied in~\cite{Hamel,Hamel1, Hamel2,CZ,LZ}, as one can find other shear flows in any $C^1$-neighborhood of the solution. In fact, \cite[Problem 1]{DE} is to show that no continuously differentiable steady 2D Euler flow is isolated in~$C^1$. In three dimensions the question is also wide open. We recall that isolated steady states are defined as follows:

\begin{definition*}
	A steady solution  $v_0$ of the incompressible Euler equations is {\em isolated in $C^1$}\/ if it is the only steady Euler flow in a $C^1$-neighborhood of~$v_0$, modulo the symmetries of the equations.
\end{definition*}

Our objective in this paper is to construct a steady 3D Euler flow which is isolated in~$C^1$. To grant us some leeway, we shall consider the equations on a closed (i.e., compact boundaryless) oriented manifold~$M$ of dimension~3 endowed with a Riemannian metric~$g$. Our main result is that indeed one can engineer this Riemannian manifold so that the Euler equations admit an isolated steady state:

\begin{theorem}\label{T.isolated}
There exists a closed $3$-dimensional Riemannian manifold $(M,g)$ of class~$C^\infty$ where the incompressible Euler equations have a smooth steady solution that is isolated in~$C^1$.
\end{theorem}

Roughly speaking, the result hinges on having a steady state with robust strongly chaotic dynamics. Thus we start off by letting~$v_0$ be an Anosov flow on~$M$ for which we can find a contact form whose Reeb field is precisely~$v_0$. This easily yields a Riemannian metric~$g$ where~$v_0$ is a Beltrami field (i.e., $\curl_g v_0=2v_0$), so that $v_0$ is, in particular, a steady Euler flow. By further tweaking the metric, one can ensure that 2~is a simple eigenvalue of the curl operator. To conclude, we then use the robust chaotic dynamics of the field to show that any $C^1$-small perturbation~$v$ of~$v_0$ that is a steady state must also satisfy the equation $\curl_g v=2v$, essentially as a consequence of the fact that~$v$ does not admit any (nonconstant) continuous first integrals. The theorem then follows from the simplicity of the eigenvalue. The proof of this theorem is given in Section~\ref{S.manifold}.  Interestingly, the most technically challenging part of the argument is the simplicity of the curl eigenvalue, so this point is presented separately in Section~\ref{S.simple}, and some parts of the proof are relegated to an Appendix. A direct consequence of the proof of Theorem~\ref{T.isolated} is that, in a suitably chosen metric, any Anosov Reeb flow on a 3-manifold defines a steady Euler flow which is isolated up to a finite-dimensional family (see Remark~\ref{R.Anosov}).

Much of this strategy carries over to the Euler equations on the usual Euclidean space setting, say on the flat torus $\TT^3:=(\RR/2\pi\ZZ)^3$. To see this, let $\cV$ be the space of eigenfields of the curl operator on~$\TT^3$ with eigenvalue~1, which is the smallest positive curl eigenvalue, see, e.g., ~\cite{PS}. It is well known that $\cV$ is a six-dimensional vector space, invariant under rigid motions. Moreover, it contains the family of {\em ABC flows}\/~\cite{AK}, i.e., the fields
\begin{equation}\label{E.uABC}
u_{ABC}(x):=\big(A\sin x_3+C\cos x_2,\; B\sin x_1+A\cos x_3,\; C\sin x_2+B\cos x_1\big)
\end{equation}
defined by arbitrary real constants $A,B$ and~$C$.

Our key observation is that one can identify a nonvanishing ABC flow~$u_0$ with positive topological entropy. Since this property is preserved by $C^1$-small perturbations and fields with  positive topological entropy cannot have a (nonconstant) analytic first integral, in Section~\ref{S.torus} we show how to establish the following weaker analog on~$\TT^3$ of Theorem~\ref{T.isolated}. For comparison, note that in a $C^k$-neighborhood of a shear flow, for any~$k$, there are infinitely many linearly independent analytic shears.

\begin{theorem}\label{T.ABC}
There exist ABC flows $u_0$ on $\mathbb T^3$ such that the only analytic steady Euler flows in a $C^1$-neighborhood of~$u_0$ belong to the six-dimensional space~$\cV$.
\end{theorem}

A minor comment is that, for the purposes of this result, there is nothing really special about ABC~flows. In fact, essentially the same argument shows that ``most'' Beltrami fields on~$\TT^3$ are isolated in~$C^1$ within the set of analytic steady Euler flows, up to the finite-dimensional family of Beltrami fields with the same curl eigenvalue. A precise statement of this fact is given in Remark~\ref{R.randomB} below.

\section{Anosov steady states: proof of Theorem~\ref{T.isolated}}
\label{S.manifold}

Key to proving Theorem~\ref{T.isolated} is the concept of Anosov flows. Recall~\cite{FH} that a vector field $X$ on a closed $3$-manifold $M$ is said to be \textit{Anosov} if there is a constant $0<\Lambda<1$ and line subbundles $E^s$ and $E^u$ of $T M$ such that $T M=E^s \oplus E^u \oplus \operatorname{span}\{X\}$, and the flow $\varphi_t$ of $X$ satisfies:

\begin{itemize}
\item  Invariance: $d \varphi_t\left(E^{s}\right)=E^{s}$ and $d \varphi_t\left(E^{u}\right)=E^{u}$ for every $t \in \mathbb{R}$.
\item Hyperbolicity: $|\left.d \varphi_t\right|_{E^s}| \leq \Lambda^t$ and $|\left.d \varphi_{-t}\right|_{E^u}| \leq \Lambda^t$ for all $t>0$. Here the operator norm $|\cdot|$ is induced by some Riemannian metric on $M$ (whose choice is irrelevant).
\end{itemize}

We shall consider Anosov flows that are also Reeb flows for some contact structure. Let $\alpha$ be a contact form on $M$, i.e., a $1$-form $\alpha$ such that $\alpha\wedge d\alpha>0$ on $M$. A Riemannian metric $g$ on $M$ is called \emph{compatible} with $\alpha$ if
$$\abs{\alpha}_g=1\,,\qquad \star_g d\alpha  = 2 \alpha\,,$$
where $\star_g$ is the Hodge star operator computed with the metric $g$, see, e.g.,~\cite{EKM}.

A contact manifold $(M,\alpha)$ defines a Reeb vector field $R$ as the unique solution to
$$i_Rd\alpha=0\,,\qquad \alpha(R)=1\,.$$
It is straightforward to check that for any compatible metric $g$ we have
$$g(R,\cdot)=\alpha(\cdot)\,,\qquad |R|_g=1\,.$$
It then follows that $R$ is a Beltrami field~\cite{PS} with constant factor $2$, and hence a steady state for the incompressible Euler equations on $(M,g)$.

The simplest example of an Anosov Reeb flow is the geodesic flow on the spherical cotangent bundle of a compact Riemannian surface of genus $>1$ with constant negative curvature, but this does not exhaust all the examples of Anosov Reeb flows~\cite{FoH}.

It is obvious that Theorem~\ref{T.isolated} is a consequence of the following result:

\begin{theorem}\label{T.Anosov}
Assume that a closed $3$-manifold $M$ admits an Anosov Reeb flow~$R$. Then there exists a Riemannian metric $g$ such that $R$ is a steady state of the incompressible Euler equations on $(M,g)$ and any other steady state $v$ that is $C^1$-close to $R$ is of the form $v=cR$ for some constant $c\in\RR$.
\end{theorem}
\begin{proof}
Let $(M,\alpha)$ be a contact $3$-manifold, and $R$ its associated Reeb field, which is Anosov by assumption. As explained above, any metric $g$ compatible with $\alpha$ makes $R$ a Beltrami field with constant factor $2$, i.e.,
\[
\curl_g R= 2R\,.
\]
By Theorem~\ref{thm: genericity}, there exists a compatible metric $g$ such that $2$ is a simple eigenvalue of $\curl_g$, and hence $R$ is the only Beltrami field with factor $2$ up to a constant multiple.

Consider a $C^1$-neighborhood $\mathcal B_\varepsilon(R)$ of $R$ and let $v\in \mathcal B_\varepsilon(R)$ be a steady solution of the incompressible Euler equations on $(M,g)$. If $\varepsilon>0$ is small enough, the stability of Anosov dynamics implies that $v$ is an Anosov flow as well~\cite{FH}. Since any Anosov volume preserving flow on a compact manifold is ergodic~\cite[Theorem 7.1.26]{FH}, any continuous first integral of $v$ must be constant on $M$. Accordingly, the Bernoulli function associated with $v$ must be constant, and $v$ then satisfies the equations
\[
\curl_g v=fv\,,\qquad \Div_g v=0\,,
\]
with factor
\[
f:=\frac{\curl_gv\cdot v}{|v|^2_g}\,.
\]
Of course, $|v|^2_g>0$ on $M$ because $v$ is $C^1$-close to the nonvanishing field $R$. Taking the divergence in the Beltrami equation above, we infer that
\[
\nabla_g f\cdot v=0\,,
\]
so $f$ is a first integral of $v$. Reasoning as above, we conclude that $f$ is constant on~$M$. Since $2$ is a simple eigenvalue of $\curl_g$, the $C^1$-closeness of $v$ to $R$ implies that $f=2$ and hence $v=cR$ for some constant $c\in\mathbb R$. This completes the proof of the theorem.
\end{proof}

\begin{remark}
A crucial property used in the proof of Theorem~\ref{T.Anosov} is the fact that Anosov flows are $C^1$-structurally stable (i.e., any $C^1$-close flow is topologically equivalent). Note that any volume preserving flow on a $3$-manifold that is $C^1$-structurally stable (or even robustly transitive, which is a weaker condition) is Anosov~\cite{BR,AM}, so the ideas in the proof of Theorem~\ref{T.Anosov} cannot be extended to deal with more general vector fields.
\end{remark}

\begin{remark}\label{R.Anosov}
Let $R$ be an Anosov Reeb flow on a closed $3$-manifold $M$. The proof of Theorem~\ref{T.Anosov} shows that~$R$ is a stationary solution of the Euler equations on $(M,g)$ for any compatible metric $g$, and it is isolated in $C^1$, up to a finite dimensional family of steady states, whose dimension is the multiplicity of the eigenvalue~2 of $\curl_g$. The most involved part of the proof of Theorem~\ref{T.Anosov} consists in proving that the compatible metric $g$ can be chosen so that the multiplicity is~$1$.
\end{remark}

\section{ABC flows: proof of Theorem~\ref{T.ABC}}
\label{S.torus}


In~\cite{Chi} it was proven that an ABC flow~\eqref{E.uABC} with $A=1$, $B\in (0,1)$ and $C>0$ small enough, which we henceforth denote by $v_0$, exhibits a transverse homoclinic intersection, which implies that the field has positive topological entropy~\cite[Theorem 5.3.5]{GH}. We also observe that $v_0$ does not vanish at any point of~$\TT^3$. Indeed, when $A=1$ and $C=0$, the ABC flow takes the form
\[
(\sin x_3,\; B\sin x_1+\cos x_3,\; B\cos x_1)\,,
\]
which is never zero provided that $B\in (0,1)$. By continuity, this property holds if $C$ is small enough, depending on~$B$.

Let $v$ be an analytic steady state in a $C^1$-neighbourhood of $v_0$. By~\eqref{E.Euler2}, the Bernoulli function $F$ is the unique solution to the equation
\[
\Delta F= \Div(v\times \curl v)\,,
\]
up to an additive constant. Since $v$ is analytic, $F$ is analytic as well. As the existence of transverse homoclinic intersections is robust under $C^1$-small perturbations, we infer that $v$ also has positive topological entropy. Then by~\cite[Corollary 4.8.5]{GH} the steady state $v$ does not admit a nontrivial analytic first integral, so $F$ is necessarily a constant on $\TT^3$.

The previous discussion shows that $v$ is a Beltrami field with factor
\[
f:=\frac{v\cdot \curl v}{|v|^2}\,,
\]
which is an analytic function because $v$ is and $v$ does not vanish at any point of $\TT^3$ provided that it is close enough to $v_0$ (which is nonvanishing). As before, $f$ is then a constant on $\TT^3$, so necessarily $f=1$, so $v$ is a curl eigenfield with eigenvalue $1$, as we wanted to prove.

\begin{remark}\label{R.randomB}
Reference~\cite{EPR23} introduced the parametric Gaussian ensemble of random Beltrami fields $(u_\lambda)$  on $\TT^3$, where the corresponding curl eigenvalue~$\lambda$ ranges over the set of admissible eigenvalues (i.e., the set of square roots of integers that are congruent with $1$, $2$, $3$, $5$ or $6$ modulo~$8$, whose multiplicity~$N_\lambda$ tends to infinity for large $\lambda$). It was proved that, with a probability that tends to~$1$ as $\lambda\to\infty$, $u_\lambda$ has positive topological entropy. The proof of Theorem~\ref{T.ABC} then shows that, with a probability that tends to~1 as $\lambda\to\infty$ along the admissible eigenvalues, the random Beltrami field $u_\lambda$ is isolated in~$C^1$ within the class of analytic steady Euler flows, up to the $N_\lambda$-dimensional family of Beltrami fields with eigenvalue~$\la$.
\end{remark}

\section{Curl generic simplicity in the class of compatible metrics}
\label{S.simple}

Let $(M,\alpha)$ be a contact $3$-manifold. We denote by $\xi:=\Ker(\alpha)$ the contact structure defined by $\alpha$. Any Riemannian metric $g$ compatible with $\alpha$ has the orthogonal decomposition
\begin{align}\label{eq.compa}
    g=g_\xi + \alpha\otimes\alpha
\end{align}
where $g_\xi$ is a smooth degenerate quadratic form on $TM$ satisfying $g_\xi(R,\cdot)=0$ that is positive definite on $\xi$.
We denote the space of smooth metrics which are compatible with $\alpha$ by $\mathcal{G}_\alpha$ and consider it as a topological subspace of the space of smooth metrics $\mathcal{G}$ on $M$. The following lemma is elementary, see, e.g.,~\cite[Proposition 2.1]{EKM}:
\begin{lemma}\label{L:triv}
A metric $g$ of the form~\eqref{eq.compa} is in $\mathcal{G}_\alpha$ if and only if its associated Riemannian volume is equal to $\frac12\alpha\wedge d\alpha$.
\end{lemma}

The main result of this section is the following theorem, which shows that for a generic choice of compatible metric, the eigenvalue $2$ of $\star_g d$ has multiplicity $1$:

\begin{theorem}\label{thm: genericity}
Let $(M,\alpha)$ be a contact $3$-manifold. There exists an open and dense subset $V$ of the space of compatible metrics $\mathcal G_\alpha$ such that for all $g\in V$ the eigenvalue $2$ of $\star_g d$ is simple.
\end{theorem}

We denote by $E_2$ the eigenspace of $\star_g d$ with eigenvalue $2$. We observe that $2$ is an eigenvalue (with a corresponding eigenform $\alpha$) for any compatible metric, so we will use the same notation $E_2$ independently of the particular compatible metric.\par

To prove Theorem~\ref{thm: genericity} one approach would be to show that $\mathcal{G}_\alpha$ is a Fr\'{e}chet manifold in its own right and to apply the techniques outlined in the Appendix. We will sidestep this technicality, viewing $\mathcal{G}_\alpha$ as a subset of the Fr\'{e}chet manifold $\mathcal{G}$ and argue in the following way: Given a compatible metric $g$ for $\alpha$, we know that $\alpha$ is contained in the eigenspace of $\star_{g} d$ associated to the eigenvalue $2$. Assume now that this eigenspace has dimension $k>1$. According to the method described in the Appendix we will then construct a codimension~1 Fr\'{e}chet-submanifold $S$ of an open set $\mathcal{U}(g) \subset\mathcal{G}$ containing $g$ such that $S$ contains all $g^\prime\in \mathcal{U}(g)$ for which the eigenvalue $\lambda=2$ of $\star_{g'}d$ has multiplicity $k$. Moreover we will construct a smooth $1$-parameter family of metrics $g_t$ with $g =g_0$ which is transverse to $S$ and stays strictly inside $\mathcal{G}_\alpha$. This proves that following $g_t$ will break up the $k$-dimensional eigenspace containing $\alpha$, thus strictly reducing its dimension, while also making sure that $\alpha$ is an eigenform for $\star_{g_t} d$ with eigenvalue $2$ for all $t$. It is clear that this idea can then be used to prove density of compatible metrics for which the eigenvalue $2$ is simple. These arguments are similar to the ones used in \cite{kepplinger2022spectral} and \cite{greilhuber2023arnolds}.

\subsection{Proof of Theorem~\ref{thm: genericity}}\label{subsection: genericity}

Let $V\subset \mathcal G_\alpha$ be the set of compatible metrics such that $2$ is a simple eigenvalue of $\star_g d$. By the continuous dependence of eigenvalues of $\star_g d$ on the Riemannian metric~\cite[Theorem 4.10]{GeraldTeschl2014} it is clear that $V$ is open, so it remains to prove that $V$ is dense in $\mathcal{G}_\alpha$. For this we want to find a perturbation within the space of compatible metrics that breaks up the multiplicity of the eigenvalue $2$.

Given $g\in \mathcal{G}_\alpha\setminus V$, the eigenspace $E_2$ has finite dimension $k>1$. We can thus pick $\beta$ in $E_2$ such that:
$$\int_M \alpha\wedge \star_g\beta=0\,\quad \text{and}\quad \int_M \beta\wedge \star_g\beta=1\,.$$
Let $\hat{M}:=\{p\in M : \alpha (p) \wedge \beta (p) \neq 0\}$ be the subset of $M$ on which $\alpha$ and $\beta$ are not collinear. The following lemma shows that is set is generic:

\begin{lemma}\label{lem: noncollinearity}
$\hat{M}$ is open and dense in $M$.
\end{lemma}
\begin{proof}
The proof of this result is part of the proof of~\cite[Lemma 3.3]{kepplinger2022spectral}.
\end{proof}

Denote by $\beta_\xi$ the projection of $\beta$ to $\xi$ along $R$, i.e.,
$$
\beta_{\xi}:=\beta-\beta(R)\alpha\,.
$$
Notice that $\beta_\xi$ is nowhere vanishing on $\hat{M}$. Let $\hat{\beta}_\xi:=\frac{\beta_\xi}{|\beta_\xi|_g}$ and a $1$-form $w$ such that
\begin{align}\label{eq.orth}
    (\alpha, \hat{\beta}_\xi,w)
\end{align}
is a positively oriented orthonormal coframe for $g$ on $\hat M$.
By construction,
\[
\alpha\wedge \hat{\beta}_\xi \wedge w=\vol_g=\frac12\alpha\wedge d\alpha\,.
\]
Without any loss of generality we shall assume that the total volume is normalized, i.e., $\int_M \vol_g=1$.

Consider the following perturbation of the metric $g$
\begin{align*}
g_\varepsilon := g + \varepsilon\, \beta_\xi \otimes \beta_\xi + \big[\big(1+\frac{\varepsilon^2}{4}|\beta_\xi|_g^4\big)^{\frac{1}{2}} -\frac{\varepsilon}{2}|\beta_\xi|_g^2- 1\big]\, g_\xi\,.
\end{align*}
It is clear that this is a metric on $M$ of the form~\eqref{eq.compa}. To compute the Riemannian volume associated to $g_\varepsilon$ we consider the following coframe,
\begin{align*}
    \bigg(\alpha,\Big[\Big(1+\frac{\varepsilon^2}{4}\,|\beta_\xi|_g^4\Big)^{\frac{1}{2}} +\frac{\varepsilon}{2}\,|\beta_\xi|_g^2\Big]^{1/2} \, \hat{\beta}_\xi,\Big[\Big(1+\frac{\varepsilon^2}{4}\,|\beta_\xi|_g^4\Big)^{\frac{1}{2}} -\frac{\varepsilon}{2}\,|\beta_\xi|_g^2\Big]^{1/2}\,w\bigg)\,.
\end{align*}
It is easy to check that it is orthonormal with respect to $g_\varepsilon$. As before, it is defined on $\hat M$. Clearly
\[
\vol_{g_\varepsilon}=\vol_g
\]
on $\hat M$, and then on the whole of $M$ because $\hat M$ is open and dense. Therefore $g_\varepsilon$ is a metric compatible with $\alpha$ by Lemma~\ref{L:triv}.

Let us now consider the linearization of the perturbation $g_\varepsilon$ at the metric $g$, i.e.,
\begin{align*}
   h:=\frac{d}{d\varepsilon}\Big|_{\varepsilon =0}\, g_\varepsilon= \beta_\xi \otimes \beta_\xi - \frac{1}{2}\,|\beta_\xi|_g^2 \, g_\xi\,,
\end{align*}
The variation of the operator $\star_g d$ with respect to $h$ is given by the following lemma~\cite[Lemma 2.1]{Enciso2012}:

\begin{lemma}\label{L:var}
  Let $h \in \mathcal{T}^{0,2}$ be a smooth symmetric $(0,2)$ tensor field and let $\alpha_1$ be an eigenform of $\star_g d$ corresponding to the eigenvalue $\lambda$. The variation $\delta_h (\star_g d)$ of $\star_g d$ in direction $h$ satisfies
\begin{align*}
   g(\alpha_2^\sharp,\big(\delta_h (\star_g d) \alpha_1\big)^\sharp)= \lambda\,h(\alpha_2^\sharp, \alpha_1^\sharp)-\frac{\lambda}{2}(\Tr_g h)\, g(\alpha_2^\sharp,\alpha_1^\sharp)
\end{align*}
for any $1$-form $\alpha_2$. As usual, $\sharp$ denotes the metric dual with respect to $g$, and $\Tr_g h$ is the trace of the tensor $h$ computed with respect to the metric $g$.
\end{lemma}

Using the orthonormal coframe~\eqref{eq.orth}, we easily get that
$$\Tr_g h=0$$
on $\hat M$, and hence on the whole $M$. Applying Lemma~\ref{L:var} to our case, we immediately obtain:
\begin{align*}
    & g(\alpha,(\delta_h (\star_g d))\alpha) = 0\,,\\
    & g(\beta,(\delta_h (\star_g d))\beta) = |\beta_\xi|_g^4\,.
\end{align*}

Let $\mathcal U(g)$ be a $C^\infty$-neighborhood of the compatible metric $g$ in $\mathcal G$. Noticing that the spectrum of $\star_g d$ is discrete and the operator $\star_g d$ defines a $C^1$ function (in the sense of Definition \ref{def: C1}) on the space of metrics, arguing as in~\cite{greilhuber2023arnolds} one can build a $C^1$ function
\begin{align*}
    \pi: \mathcal{U}(g) \to  \sym(k,\mathbb R)\,,
\end{align*}
where $\sym(k, \mathbb R)$ is the space of $k\times k$ symmetric matrices with real entries. The definition of $\pi$ is rather involved and can be found in the aforementioned reference, but its essential properties are summarized in the Appendix. Crucially, the subset of metrics in $\mathcal{U}_g$ for which the $k$-dimensional eigenspace $E_2$ does not split up is contained in
$$\pi^{-1}(\mathbb R \cdot \id )\,.$$

Next, we claim that if there exists a variation of the metric $h\in T_g \mathcal{G}$ such that the derivative $\pi^\prime(h)$ of $\pi$ applied to $h$ is not contained in $\mathbb R \cdot \id $, then the set $\pi^{-1}(\mathbb R \cdot \id )$ is contained in a codimension $1$ Fr\'{e}chet-submanifold of some neighborhood of $g$ in $\mathcal{U}_g$. Indeed, set $v:=\pi^\prime (h)$ and take any linear projection $p_v: \sym(k,\mathbb R)\to \mathbb R \cdot v$ to the span of $v$ that contains $\mathbb R \cdot \id $ in its kernel (this obviously exists by the choice of $h$). The composition
\begin{align*}
    p_v\circ \pi: \mathcal{U}_g \to \mathbb R
\end{align*}
is a $C^1$ map in the sense of Hamilton whose zero set contains $\pi^{-1}(\mathbb R\cdot \id )$. Moreover, the derivative of $p_v\circ \pi$ is $p_v\circ \pi'$ so it follows by construction that $(p_v\circ \pi')(h)=v\neq 0$, so the map $p_v\circ \pi$ is a submersion. We can thus apply Proposition~\ref{proposition: submersion} in the Appendix to conclude that, after possibly restricting $\mathcal{U}_g$, all metrics in $\pi^{-1}(\mathbb R\cdot \id )$ are contained in a codimension $1$ Fr\'{e}chet submanifold of $\mathcal{U}_g$.

Now we prove that the variation $h$ we constructed above satisfies these requirements, i.e., $\pi^\prime (h)$ is not a multiple of the identity matrix. Extend the pair $\alpha$ and $\beta$ to an $L^2$-orthonormal basis $(u_i)_{i=1}^{k}$ of $E_2$ where $u_1 = \alpha$ and $u_2 = \beta$. Theorem~\ref{thm: general genericity theorem} tells us the expression of the map $\pi^\prime$:
\begin{align*}
    \pi^\prime: h\mapsto \big(\langle u_i, \delta_h (\star_g d) u_j \rangle_{L^2} \big)_{i,j}\,.
\end{align*}
If $\pi^\prime (h)$ were contained in the span of the identity matrix then, in particular,
$$
-\langle u_1, \delta_h (\star_g d) u_1 \rangle_{L^2}+\langle u_2, \delta_h (\star_g d) u_2 \rangle_{L^2}=0\,.
$$
Taking the variation $h$ we constructed above, we obtain
\begin{align*}
-\langle \alpha, \delta_h (\star_g d) \alpha \rangle_{L^2}+\langle \beta, \delta_h (\star_g d) \beta \rangle_{L^2} = \int_M |\beta_\xi |_g^4 \,\vol_g\,,
\end{align*}
which cannot be $0$ by Lemma~\ref{lem: noncollinearity}.

The previous arguments show that there exists a neighborhood $\mathcal{U}(g)$ of $g$ so that the set of compatible metrics in $\mathcal{U}(g)$ for which the eigenspace $E_2$ has dimension $k$ is contained in a Fr\'{e}chet submanifold $S$ of codimension $1$, and furthermore the $1$-parameter family of metrics $g_\varepsilon$ is transverse to $S$, that is $g_\varepsilon$ is not contained in $S$ for all $\varepsilon \neq 0$ provided $\lvert \varepsilon \rvert$ is small enough. Since the dimension of eigenspaces is lower semi-continuous, by perturbing in direction $h$ we therefore get a compatible metric $g^\prime$ for which $E_2$ has dimension strictly less than $k$, and after repeating the above argument finitely many times we end up with a $1$-dimensional eigenspace. As the perturbations are arbitrarily small, we get density of compatible metrics for which the eigenspace $E_2$ is simple.

\begin{remark}
It was proven in \cite{kepplinger2022spectral} that the curl operator satisfies spectral simplicity along generic $1$-parameter families of Riemannian metrics, and it is natural to ask whether the same holds in this setting. Using the example of Berger spheres, however, it is easy to see that generic simplicity is the best one can hope for in this setting.
\end{remark}

\section*{Acknowledgements}

The authors are very grateful to Theo Drivas for his interest in this work and useful remarks. This work has received funding from the European Research Council (ERC) under the European Union's Horizon 2020 research and innovation programme through the grant agreement~862342 (A.E.). It is partially supported by the grants CEX2023-001347-S, RED2022-134301-T and PID2022-136795NB-I00 (A.E. and D.P.-S.) funded by MCIN/AEI/10.13039/501100011033, and Ayudas Fundaci\'on BBVA a Proyectos de Investigaci\'on Cient\'ifica 2021 (D.P.-S.).

W.K. thanks his advisors Vera V\'ertesi and Michael Eichmair for their constant support and helpful mentoring as well as the Vienna School of Mathematics for providing a pleasant research environment. This research was funded in part by the Austrian Science Fund (FWF) [10.55776/P34318] and [10.55776/Y963]. For open access purposes, the author has applied a CC BY public copyright license to any author-accepted manuscript version arising from this submission.

\appendix

\section{Splitting of eigenvalues}
\label{A.splitting}

For families of self-adjoint operators whose spectrum consists of discrete eigenvalues of finite multiplicity there is a powerful tool, first introduced by Colin de Verdi\`{e}re~\cite{CdV1988} and developed further in~\cite{Teytel1999}, \cite{kepplinger2022spectral}, as well as in \cite{greilhuber2023arnolds}, to analyze the splitting behavior of multiple eigenvalues. We will outline the main idea of this technique after fixing some notation and defining the following notion of differentiability introduced by Hamilton~\cite{Hamilton1982}.

\begin{definition}\label{def: C1}
A map $F:\mathcal{F}\to \mathcal{F}^\prime$ between Fr\'{e}chet spaces $\mathcal{F}$ and $\mathcal{F}^\prime$ is called $C^1$ if the directional derivative
\begin{align*}
    DF(x)u = \lim_{t\to 0} \frac{F(x + tu) - F(x)}{t}
\end{align*}
exists for all $x \in \mathcal F$, $u \in T_x\mathcal F$ and, furthermore, $DF(x)u$ is linear in $u$ and jointly continuous in $x$ and $u$.
\end{definition}

Let $\mathcal{X}$ be a Fr\'{e}chet manifold, $H$ a Banach space, $\left\{\langle \,\cdot\,,\,\cdot\,\rangle_q\right\}_{q\in \mathcal{X}}$ a $C^1$ family of inner products on $H$, and $\{A_q\}_{q\in \mathcal{X}}$ a $C^1$ family of unbounded operators from some fixed dense domain $D\subset H$ to $H$, each self-adjoint with respect to $\langle \,\cdot\,,\,\cdot\,\rangle_q$ and with compact resolvent. Suppose that $\lambda$ is an eigenvalue with multiplicity $k$ of $A_{q_0}$ for some $q_0 \in \mathcal X$.\par
The compactness of the resolvent of $A_{q_0}$ implies that $\lambda$ is isolated and we can encircle it by some simple closed loop $\gamma$ in $\mathbb C$ which does not enclose any other eigenvalues of $A_{q_0}$. By the variational characterization of eigenvalues of self-adjoint operators \cite{GeraldTeschl2014}, the eigenvalues of $A_q$ move ``continuously'' with respect to the parameter $q$ in the sense that there exists a neighborhood $\mathcal{U}(q_0)$ of $q_0$ such that the number of eigenvalues of $A_q$, when counted with multiplicity, encircled by $\gamma$ is exactly $k$ for all $q\in \mathcal{U}(q_0)$. One can associate to $\gamma$ a spectral projection $P_\gamma (q)$
\begin{align*}
    P_\gamma(q):=\frac{1}{2\pi i}\int_\gamma (z - A_q)^{-1} dz\,,
\end{align*}
which maps the Banach space $H$ to the sum of eigenspaces associated to the eigenvalues of $A_q$ encircled by $\gamma$. We will denote the image set of $P_\gamma (q)$ by $E(q)$. The fact that the family of operators is $C^1$ implies that $P_\gamma$ is $C^1$ as well, see the proof of \cite[Lemma C.1]{greilhuber2023arnolds}. One can restrict $P_\gamma (q)$ to a linear bijection
\begin{align*}
P_\gamma (q)\vert_{E(q_0)} : E(q_0) \to E(q)\,,
\end{align*}
which is invertible, and this family of inverses $(q\mapsto P_\gamma (q) \vert_{E(q_0)}^{-1})$ is $C^1$ as well \cite[Theorem II.3.1.]{Hamilton1982}. Using this we can build the $C^1$ function
\begin{align*}
\tilde{\pi}=(P_\gamma \vert_{E(q_0)})^{-1} \circ A_{(\,\cdot\,)} \circ P_\gamma\vert_{E(q_0)} : \mathcal{U}(q_0) \to \End (E(q_0))\,.
\end{align*}
With some more work $\tilde{\pi}$ can be modified to a different map $\pi$ such that $\pi(q)$ is a symmetric endomorphism on $E(q_0)$ with respect to $\langle \,\cdot\,,\,\cdot\,\rangle_{q_0}$ for all $q\in \mathcal{U}(q_0)$. This involves a parametric version of the Gram-Schmidt orthonormalization procedure and is done in detail in \cite[Section 2.1]{greilhuber2023arnolds}.
The final result is a map
\begin{align*}
\pi: \mathcal{U}(q_0)\to \sym(E(q_0))\cong \sym(k,\mathbb R)\,,
\end{align*}
where we denote the symmetric endomorphisms of $E(q_0)$ with respect to $\langle \,\cdot\,,\,\cdot\,\rangle_{q_0}$ by $\sym(E(q_0))$ (i.e., the space of $k\times k$ symmetric matrices with real entries). The essential properties of $\pi$ are summarized in the following Theorem.

\begin{theorem}[\protect{\cite[Theorem $2.1$]{greilhuber2023arnolds}}]\label{thm: general genericity theorem}
Let $\mathcal{X}$ be a Fr\'{e}chet manifold, $H$ a Banach space, $\left\{\langle \,\cdot\,,\,\cdot\,\rangle_q\right\}_{q\in \mathcal{X}}$ a $C^1$ family of inner products on $H$, and $\{A_q\}_{q\in \mathcal{X}}$ a $C^1$ family of unbounded operators from some fixed dense domain $D\subset H$ to $H$, each self-adjoint with respect to $\langle \,\cdot\,,\,\cdot\,\rangle_q$ and with compact resolvent. Suppose $\lambda$ is an eigenvalue with multiplicity $k$ of $A_{q_0}$ for some $q_0 \in \mathcal X$.\par
Then, for $\varepsilon>0$ small enough, there exist an open neighborhood $\mathcal{U} (q_0)$ of $q_0$ and a $C^1$ function $\pi: \mathcal{U} (q_0) \to \sym(k,\mathbb R)$ satisfying
    \begin{align*}
        \sigma (A_{q})\cap (\lambda-\varepsilon,\lambda+\varepsilon) = \sigma (\pi (q))
    \end{align*}
for all $q\in \mathcal{U}(q_0)$, where $\sigma$ denotes the spectrum of the corresponding operator. In particular, $\pi^{-1} (\mathbb R \cdot \id )$ is the set of parameters $q\in \mathcal{U}(q_0)$ such that $\lambda (q)$ has multiplicity $k$. Furthermore the derivative of
$\pi$ at $q_0$ is given by
\begin{align*}
       \pi^\prime: T_{q_0}\mathcal{U}&(q_0) \to T_{\id } \sym(k,\mathbb R) \cong \sym(k,\mathbb R)\,, \\
       & h \mapsto  \big( \left\langle DA_{q_0} [h] u_m, u_\ell \right\rangle \big)_{m,\ell = 1}^k\,,
\end{align*}
where $\{u_m\}_{m=1}^k$ is an orthonormal basis of the $\lambda$-eigenspace of $A_{q_0}$.
\end{theorem}

This result means that the map $\pi$ measures how the multiple eigenvalue $\lambda$ breaks up as $A_{q_0}$ is perturbed. Remarkably, a full understanding of $\pi$ gives a precise picture of this behaviour, not just one up to first order. Nevertheless it is of great interest to understand the derivative of $\pi$ at $q_0$. In Section~\ref{subsection: genericity} we have used the derivative of $\pi$ together with the following proposition to construct a submanifold containing all those parameters $q$ in some open neighborhood of $q_0$ for which the $k$-dimensional eigenspace corresponding to $\lambda$ is not broken up.

\begin{proposition}[\protect{\cite[Proposition $A.2$]{greilhuber2023arnolds}}]
    \label{proposition: submersion}
    Let $\mathcal U \subseteq \mathcal F$ be an open subset of a Fr\'{e}chet space $\mathcal F$. Suppose a $C^1$ map $F: \mathcal U \to \mathbb R^m$ is such that $DF(p): \mathcal F \to \mathbb R^m$ is surjective at a point $p \in \mathcal U \cap F^{-1}(0)$. Then there exists a neighborhood $\mathcal U' \subseteq \mathcal U$ of $p$, an open neighborhood $\mathcal V \subseteq \mathcal F'$ of the origin in some Fr\'{e}chet space $\mathcal F'$, and a $C^1$ map $G:\mathcal V \to \mathcal U'$ such that $F^{-1}(0) \cap \mathcal U'= G(\mathcal V)$.
\end{proposition}

Proposition \ref{proposition: submersion} implies that preimages of regular values of $C^1$ maps from Fr\'{e}chet manifolds to $\mathbb R^m$ are submanifolds of codimension $m$. It is noteworthy that in this finite codimension case one can prove Proposition \ref{proposition: submersion} and the associated implicit function theorem \cite[Lemma A.1]{greilhuber2023arnolds} without using Nash--Moser iteration.

\bibliographystyle{amsplain}

\end{document}